\documentclass{article}

\usepackage[french,british]{babel}

\usepackage[T1]{fontenc}

\usepackage{amsmath}

\usepackage{amsfonts}
\usepackage{amsbsy}
\usepackage{amssymb}

\usepackage{amsthm}

\theoremstyle{plain}

\usepackage{latexsym}
\usepackage{amsmath}
\usepackage{amsfonts}
\usepackage{amssymb}
\usepackage{psfrag}
\usepackage{graphicx}

\usepackage{amsmath,amssymb}
\usepackage{graphicx}

\usepackage{amsmath,amssymb}
\usepackage{graphicx}

\newtheorem{ozn}{Definition}[section]
\newtheorem{thm}{Theorem}[section]

\newtheorem{nas}{Corollary}[section]

\newtheorem{lema}{Lemma}[section]

\newcommand{\me}{\mathbf}
\newcommand{\mr}{\mathbb}
\newcommand{\mt}{\mathsf}
\newcommand{\md}{\mathcal}
\newcommand{\ld}{\left}
\newcommand{\rd}{\right}

\newcommand{\ip}{\int_{-\pi}^{\pi}}

\newcommand{\be}{\begin{equation}}
\newcommand{\ee}{\end{equation}}

\newcommand{\bem}{\begin{multline}}
\newcommand{\eem}{\end{multline}}

\newcommand{\bml}{\begin{multline*}}
\newcommand{\eml}{\end{multline*}}

\newcommand{\beg}{\begin{gather}}
\newcommand{\eeg}{\end{gather}}

\begin{document}

\title{Minimax interpolation of continuous time stochastic processes with periodically correlated increments observed with noise}

\author{
Maksym Luz\thanks {BNP Paribas Cardif in Ukraine, Kyiv, Ukraine, maksym.luz@gmail.com},
Mikhail Moklyachuk\thanks
{Department of Probability Theory, Statistics and Actuarial
Mathematics, Taras Shevchenko National University of Kyiv, Kyiv 01601, Ukraine, moklyachuk@gmail.com}
}

\date{\today}

\maketitle

\renewcommand{\abstractname}{Abstract}
\begin{abstract}
We deal with the problem of optimal estimation of the linear   functionals
constructed from the missed values   of a continuous time  stochastic process $\xi(t)$ with periodically stationary increments at points $t\in[0;(N+1)T]$ based on  observations of this process with periodically stationary noise. To solve the problem,
a sequence of stochastic functions
$
\{\xi^{(d)}_j(u)=\xi^{(d)}_j(u+jT,\tau),\,\, u\in [0,T), j\in\mathbb Z\}.
$ is constructed. It    forms a $L_2([0,T);H)$-valued
stationary increment sequence $\{\xi^{(d)}_j,j\in\mathbb Z\}$ or   corresponding to it
an (infinite dimensional) vector stationary increment sequence $\{\vec\xi^{(d)}_j=(\xi^{(d)}_{kj}, k=1,2,\dots)^{\top},
j\in\mathbb Z\}$.
In the case of a known spectral density, we obtain formulas for calculating values of the mean square errors and the spectral characteristics of the optimal estimates of the functionals.
Formulas determining  the least favorable spectral densities and the minimax (robust) spectral
characteristics of the optimal linear estimates of functionals
are derived in the case where the  sets of admissible spectral densities are given.
\end{abstract}

\maketitle


\maketitle

\vspace{2ex}

\section*{Introduction}

In this article, we investigate the interpolation problem for the stochastic processes $\xi(t)$, $t\in\mr R$, with a periodically
correlated increments $\xi^{(d)}(t,\tau T)=\Delta_{T\tau}^d \xi(t)$ of    order $d$ and period $T$, where $\Delta_{s} \xi(t)=\xi(t)-\xi(t-s)$. The resent studies, for example  Basawa et al. \cite{Basawa}, Dudek et al. \cite{Dudek-Hurd}, Reisen et al. \cite{Reisen2018}, show a  constant interest to the non-stationary models and robust methods of estimation.

Kolmogorov \cite{Kolmogorov},  Wiener \cite{bookWien} and  Yaglom \cite{Yaglom} developed the methods of solution  of interpolation, extrapolation   and filtering problems for stationary stochastic sequences and processes. For these problems, they considered the   estimates $\widetilde{x}(t)$ constructed as a linear combination of the available observations.  As an  optimal linear estimate, they chose the one which  minimizes the mean square error $\Delta(\widetilde{x}(t),f)=\mt E|x(t)-\widetilde{x}(t)|^2$ for the fixed spectral density $f(\lambda)$ of the stationary process or sequence $x(t)$. The problems were also studied in the presence of the noise sequence or process.

The developed  classical estimation methods are not directly applicable in practice. Usually, the exact    spectral structure of the processes isn't  available. The estimated spectral densities can be considered instead. However, Vastola and Poor \cite{VastPoor1983} showed the concrete examples, where such substitution  can result in a significant increase of  the   estimate error.  Therefore, it is reasonable to consider the estimates, which minimize the maximum of the mean-square errors for all spectral densities from a given set of admissible spectral densities simultaneously. The minimax-robust estimation method  was proposed by Grenander \cite{Grenander} who considered an extrapolation of the functional $Ax = \int_{0}^1 a (t) x (t) dt$ as a  game between two players, one of which minimizes $\Delta(\widetilde{A}x,f)$ by $\widetilde{A}\zeta$ and another one maximizes it by $f$. The game has a saddle point under proper conditions:
$$ \max_{\substack{f}} \min_{\substack{\widetilde{A}x}} \mt \Delta(\widetilde{A}x,f) = \min_{\substack{\widetilde{A}x}} \max_{\substack {f}} \Delta(\widetilde{A}x,f) =\nu. 
$$
  For more details, we refer to the further study by Franke and Poor \cite{Franke1984} and  the survey paper by Kassam and Poor \cite{Kassam_Poor1985}. A wide range of results has been obtained by  Moklyachuk \cite{Moklyachuk:1981,Moklyachuk,Moklyachuk2015}. These results have been  extended on the vector-valued stationary processes and sequences by Moklyachuk and Masyutka \cite{Mokl_Mas_book}.

The concept of stationarity   admits some generalizations,  two of which are  stationary $d$th increments and periodical  correlation. A combination of them is in scope of the article.    Random processes with stationary $d$th increments $x(t)$ were introduced   by Yaglom and Pinsker \cite{Pinsker}. The increment process $x^{(d)}(t,\tau ):=\Delta_{\tau}^d x(t)$ generated by $x(t)$ is stationary by the variable $t$, which means that the mathematical expectations $\mt E\xi^{(n)}(t,\tau)$ and $\mt E\xi^{(n)}(t+s,\tau_1)\xi^{(n)}(t,\tau_2)$ do not depend on $t$.
Yaglom and Pinsker described the spectral representation of such process and the spectral density canonical factorization. They also introduced and solved the extrapolation problem for them.
 The minimax-robust extrapolation, interpolation and filtering  problems for stochastic  processes with stationary increments were investigated by Luz and Moklyachuk \cite{Luz_Mokl_book}.

Dubovetska and Moklyachuk \cite{Dubov1}  derived the  classical and minimax-robust estimates for another generalization of stationary processes --  periodically  correlated (cyclostationary) processes, introduced by Gladyshev \cite{Glad1963}. The correlation function  $K(t,s)={\textsf{E}}{x(t)\overline{x(s)}}$ of such processes is a $T$-periodic function: $K(t,s)=K(t+T,s+T)$, which implies a time-dependent spectrum. Periodically  correlated  processes are widely used in signal processing and communications (see Napolitano \cite{Napolitano} for a review of recent works on cyclostationarity and its applications).

In this article, we deal with a problem of the mean-square optimal estimation of the linear functional    $A_{NT}{\xi}=\int_{0}^{(N+1)T}a(t)\xi(t)dt$
which depends on the
unknown values of a random process $\xi(t)$ with periodically stationary $d$th
increments from  observations $\xi(t)+\eta(t)$ at points
$\mr R\setminus[0;(N+1)T]$. The similar problems for discrete time processes have been studied by Kozak and Moklyachuk \cite{Kozak_Mokl}, Luz and Moklyachuk \cite{Luz_Mokl_int_noise_GMI,Luz_Mokl_extra_noise_PCI}. The extrapolation problem without noise  for continues time processes with periodically stationary
increments has been studied by Luz and Moklyachuk \cite{Luz_Mokl_extra_cont_PCI}.
   In sections \ref{section_prelim_i_n_c}, we describe the periodically stationary increment process as a  stationary $H$-valued increment sequence.
In   section \ref{classical_interpolation_i_n_c}, the classical interpolation problem is introduced and solved. Particularly, formulas for calculating  the mean-square error and the spectral characteristic of the optimal linear
estimates of the functional $A_{NT}{\xi}$ are derived under the condition of spectral certainty.
 The results on minimax-robust interpolation for the studied processes are presented in   section \ref{minimax_estimation_i_n_c}, where the  relations that determine the least favourable spectral densities and the minimax spectral
characteristics are derived for some classes of spectral densities.

\section{Preliminary results}
\label{section_prelim_i_n_c}
\subsection{Periodically correlated processes and generated vector stationary sequences}
\label{section_PC_process_i_n_c}
In this section, we present a brief review of periodically correlated processes and describe an approach of presenting it as a  stationary $H$-valued   sequence. I n the next section, this approach is applied  to develop a spectral theory for periodically correlated increment process.

\begin{ozn}[Gladyshev \cite{Glad1963}] A mean-square continuous stochastic process
$\eta:\mathbb R\to H=L_2(\Omega,\mathcal F,\mathbb P)$, with $\textsf {E} \eta(t)=0,$ is called periodically correlated (PC) with period  $T$, if its correlation function  $K(t,s)={\textsf{E}}{\eta(t)\overline{\eta(s)}}$  for all  $t,s\in\mathbb R$ and some fixed $T>0$ is such that
\[
K(t,s)={\textsf{E}}{\eta(t)\overline{\eta(s)}}={\textsf{E}}{\eta(t+T)\overline{\eta(s+T)}}=K(t+T,s+T).
\]
\end{ozn}

For a periodically correlated  stochastic process $\eta(t)$, one can construct
the following sequence of stochastic functions \cite{DubovetskaMoklyachuk2013}, \cite{MoklyachukGolichenko2016}
\begin{equation} \label{zj}
\{\eta_j(u)=\eta(u+jT),u\in [0,T), j\in\mathbb Z\}.
\end{equation}
Sequence (\ref{zj}) forms a $L_2([0,T);H)$-valued
stationary sequence $\{\eta_j,j\in\mathbb Z\}$ with the correlation function
\begin{eqnarray*}
B_{\eta}(l,j) = \langle\eta_l,\eta_j\rangle_H &=&\int_0^T
\textsf{E}[\eta(u+lT)\overline{\eta(u+jT)}]du
\\
&=&\int_0^TK_{\eta}(u+(l-j)T,u)du =
 B_{\eta}(l-j),
\end{eqnarray*}
where
$K_{\eta}(t,s)=\textsf{E}{\eta}(t)\overline{{\eta}(s)}$
is the correlation function of the PC process $\eta(t)$.
Chose the following orthonormal basis in the space $L_2([0,T);\mathbb{R})$
\be\label{orthonormal_basis}
\{\widetilde{e}_k=\frac{1}{\sqrt{T}}e^{2\pi
i\{(-1)^k\left[\frac{k}{2}\right]\}u/T}, k=1,2,3,\dots\}, \quad
\langle \widetilde{e}_j,\widetilde{e}_k\rangle=\delta_{kj}.
\ee
Then the stationary sequence $\{\zeta_j,j\in\mathbb Z\}$  can be represented in the form
\be\label{eta_cont_i_n_c}
\eta_j= \sum_{k=1}^\infty \eta_{kj}\widetilde{e}_k,
\ee
where
\[\eta_{kj}=\langle\eta_j,\widetilde{e}_k\rangle =
\frac{1}{\sqrt{T}} \int_0^T \eta_j(v)e^{-2\pi
i\{(-1)^k\left[\frac{k}{2}\right]\}v/T}dv.\]
The sequence
$\{\zeta_j,j\in\mathbb Z\},$
   or a corresponding to it  vector sequence
  \[\{\vec \eta_j=(\eta_{kj}, k=1,2,\dots)^{\top},
j\in\mathbb Z\},\]
is called a generated by the process $\{\eta(t),t\in\mathbb R\}$ vector stationary sequence.
The components $\{\eta_{kj}\}: k=1,2,\dots;j\in\mathbb Z$  of the generated stationary sequence
$\{\eta_j,j\in\mathbb Z\}$ satisfy the  relations \cite{Kallianpur}, \cite{Moklyachuk:1981}
\[
\textsf{E}{\eta_{kj}}=0, \quad \|{\eta}_j\|^2_H=\sum_{k=1}^\infty
\textsf{E}|\eta_{kj}|^2\leq P_\eta=B_\eta(0), \quad
\textsf{E}\eta_{kl}\overline{\eta}_{nj}=\langle
R_{\eta}(l-j)\widetilde{e}_k,\widetilde{e}_n\rangle.
\]
The correlation function $R_{\eta}(j)$  of the generated stationary
sequence $\{\eta_j,j\in\mathbb Z\}$
 is a correlation operator function.
 The correlation operator $R_{\eta}(0)=R_{\eta}$ is a
kernel operator and its kernel norm satisfies the following limitations:
\[
\|{\eta}_j\|^2_H=\sum_{k=1}^\infty \langle R_{\eta}
\widetilde{e}_k,\widetilde{e}_k\rangle\leq P_\eta. \]
The generated stationary sequence $\{\eta_j,j\in\mathbb Z\}$ has the spectral density function
$g(\lambda)=\{g_{kn}(\lambda)\}_{k,n=1}^\infty$, that is positive valued operator  functions of variable
 $\lambda\in [-\pi,\pi)$, if its correlation function $R_{\eta}(j)$ can be represented in the form
\[
\langle R_{\eta}(j)\widetilde{e}_k,\widetilde{e}_n\rangle=\frac{1}{2\pi} \int _{-\pi}^{\pi}
e^{ij\lambda}\langle g(\lambda) \widetilde{e}_k,\widetilde{e}_n\rangle d\lambda.
\]
We finish our review by the statement, that
for almost all   $\lambda\in [-\pi,\pi)$ the spectral density $f(\lambda)$ is a kernel operator with an integrable kernel norm
\[
\sum_{k=1}^\infty \frac{1}{2\pi} \int _{-\pi}^\pi \langle g(\lambda)
\widetilde{e}_k,\widetilde{e}_k\rangle d\lambda=\sum_{k=1}^\infty\langle R_{\eta}
\widetilde{e}_k,\widetilde{e}_k\rangle=\|{\eta}_j\|^2_H\leq P_\eta.
\]

\subsection{Stochastic processes with periodically correlated $d$th increments}
\label{section_PC_increment_i_n_c}

For a given stochastic process $\{\xi(t),t\in\mathbb R\}$,   consider the stochastic $d$th increment process
\begin{equation}
\label{oznachPryrostu_cont}
\xi^{(d)}(t,\tau)=(1-B_{\tau})^d\xi(t)=\sum_{l=0}^d(-1)^l{d \choose l}\xi(t-l\tau),
\end{equation}
with the step $\tau\in\mr R$, generated by the stochastic process $\xi(t)$. Here $B_{\tau}$ is the backward shift operator: $B_{\tau}\xi(t)=\xi(t-\tau)$, $\tau\in \mr R$.

We prefer to use the notation $\xi^{(d)}(t,\tau)$ instead of widely used $\Delta_{\tau}^{d}\xi(t)$ to avoid a duplicate with the mean square error notation.

\begin{ozn}[Luz and Moklyachuk \cite{Luz_Mokl_extra_cont_PCI}]
\label{OznPeriodProc2_cont}
A stochastic process $\{\xi(t),t\in\mathbb R\}$ is called a \emph{stochastic
process with periodically stationary (periodically correlated) increments} with the step $\tau\in\mr Z$ and the period $T>0$ if the mathematical expectations
\begin{eqnarray*}
\mt E\xi^{(d)}(t+T,\tau T) & = & \mt E\xi^{(d)}(t,\tau T)=c^{(d)}(t,\tau T),
\\
\mt E\xi^{(d)}(t+T,\tau_1 T)\xi^{(d)}(s+T,\tau_2 T)
& = & D^{(d)}(t+T,s+T;\tau_1T,v_2T)
\\
& = & D^{(d)}(t,s;\tau_1T,\tau_2T)
\end{eqnarray*}
exist for every  $t,s\in \mr R$, $\tau_1,\tau_2 \in \mr Z$ and for some fixed $T>0$.
\end{ozn}

The functions $c^{(d)}(t,\tau T)$ and  $D^{(d)}(t,s;\tau_1T,\tau_2 T)$ from  Definition \ref{OznPeriodProc2_cont} are called the \emph{mean value} and  the \emph{structural function} of the stochastic
process $\xi(t)$ with periodically stationary (periodically correlated) increments.

For the stochastic process $\{\xi(t), t\in \mathbb R\}$ with periodically correlated  increments $\xi^{(d)}(t,\tau T)$ and the integer step $\tau$, we follow the procedure described in  Section \ref{section_PC_process_i_n_c} and construct
a sequence of stochastic functions
\begin{equation} \label{xj}
\{\xi^{(d)}_j(u):=\xi^{(d)}_{j,\tau}(u)=\xi^{(d)}_j(u+jT,\tau T),\,\, u\in [0,T), j\in\mathbb Z\}.
\end{equation}

Sequence (\ref{xj}) forms a $L_2([0,T);H)$-valued
stationary increment sequence $\{\xi^{(d)}_j,j\in\mathbb Z\}$ with the structural function
\begin{eqnarray*}
B_{\xi^{(d)}}(l,j)&= &\langle\xi^{(d)}_l,\xi^{(d)}_j\rangle_H =\int_0^T
\textsf{E}[\xi^{(d)}_j(u+lT,\tau_1 T)\overline{\xi^{(d)}_j(u+jT,\tau_2 T)}]du
\\&=&\int_0^T D^{(d)}(u+(l-j)T,u;\tau_1T,\tau_2T) du =
 B_{\xi^{(d)}}(l-j).
 \end{eqnarray*}
Making use of the orthonormal basis \eqref{orthonormal_basis} the stationary increment sequence  $\{\xi^{(d)}_j,j\in\mathbb Z\}$ can be represented in the form
\begin{equation} \label{zeta_cont_i_n_c}
\xi^{(d)}_j= \sum_{k=1}^\infty \xi^{(d)}_{kj}\widetilde{e}_k,\end{equation}
where
\[\xi^{(d)}_{kj}=\langle\xi^{(d)}_j,\widetilde{e}_k\rangle =
\frac{1}{\sqrt{T}} \int_0^T \xi^{(d)}_j(v)e^{-2\pi
i\{(-1)^k\left[\frac{k}{2}\right]\}v/T}dv.\]

We call this sequence
$\{\xi^{(d)}_j,j\in\mathbb Z\},$
   or the corresponding to it vector sequence
  \be \label{generted_incr_seq_i_n_c}
  \{\vec\xi^{(d)}(j,\tau)=\vec\xi^{(d)}_j=(\xi^{(d)}_{kj}, k=1,2,\dots)^{\top}=(\xi^{(d)}_{k}(j,\tau), k=1,2,\dots)^{\top},
j\in\mathbb Z\},\ee
\emph{an infinite dimension vector stationary increment sequence} generated by the increment process $\{\xi^{(d)}(t,\tau T),t\in\mathbb R\}$.
Further, we will omit the word vector in the notion generated vector stationary increment sequence.

Components $\{\xi^{(d)}_{kj}\}: k=1,2,\dots;j\in\mathbb Z$  of the generated stationary increment sequence
$\{\xi^{(d)}_j,j\in\mathbb Z\}$ are such that, \cite{Kallianpur}, \cite{Moklyachuk:1981}
\[
\textsf{E}{\xi^{(d)}_{kj}}=0, \quad \|\xi^{(d)}_j\|^2_H=\sum_{k=1}^\infty
\textsf{E}|\xi^{(d)}_{kj}|^2\leq P_{\xi^{(d)}}=B_{\xi^{(d)}}(0),
\]
and
\[\textsf{E}\xi^{(d)}_{kl}\overline{\xi^{(d)}_{nj}}=\langle
R_{\xi^{(d)}}(l-j;\tau_1,\tau_2)\widetilde{e}_k,\widetilde{e}_n\rangle.
\]
The \emph{structural function}  $R_{\xi^{(d)}}(j):=R_{\xi^{(d)}}(j;\tau_1,\tau_2)$  of the generated stationary increment
sequence $\{\xi^{(d)}_j,j\in\mathbb Z\}$
 is a correlation operator function.
 The correlation operator $R_{\xi^{(d)}}(0)=R_{\xi^{(d)}}$ is a
kernel operator and its kernel norm satisfies the following limitations:
\[
\|\xi^{(d)}_j\|^2_H=\sum_{k=1}^\infty \langle R_{\xi^{(d)}}
\widetilde{e}_k,\widetilde{e}_k\rangle\leq P_{\xi^{(d)}}. \]

 Suppose that  the structural function $R_{\xi^{(d)}}(j)$ admits a representation
\[
\langle R_{\xi^{(d)}}(j;\tau_1, \tau_2)\widetilde{e}_k,\widetilde{e}_n\rangle=\frac{1}{2\pi} \int _{-\pi}^{\pi}
e^{ij\lambda}(1-e^{-i\tau_1\lambda})^d(1-e^{i\tau_2\lambda})^d\frac{1}
{\lambda^{2d}}\langle f(\lambda) \widetilde{e}_k,\widetilde{e}_n\rangle d\lambda.
\]
Then
$f(\lambda)=\{f_{kn}(\lambda)\}_{k,n=1}^\infty$ is a \emph{spectral density function} of the generated stationary increment sequence $\{\xi^{(d)}_j,j\in\mathbb Z\}$. It  is a positive valued operator  functions of variable
 $\lambda\in [-\pi,\pi)$, and for almost all   $\lambda\in [-\pi,\pi)$ it is a kernel operator with an integrable kernel norm
\begin{multline} \label{P1}
\sum_{k=1}^\infty \frac{1}{2\pi} \int _{-\pi}^\pi (1-e^{-i\tau_1\lambda})^d(1-e^{i\tau_2\lambda})^d\frac{1}
{\lambda^{2d}} \langle f(\lambda)
\widetilde{e}_k,\widetilde{e}_k\rangle d\lambda=
\\
=\sum_{k=1}^\infty\langle R_{\xi^{(d)}}
\widetilde{e}_k,\widetilde{e}_k\rangle=\|{\zeta}_j\|^2_H\leq P_{\xi^{(d)}}.
\end{multline}

The stationary $d$th increment sequence $\vec{\xi}^{(d)}_j$ admits the spectral representation
\[
\vec{\xi}^{(d)}_j=\int_{-\pi}^{\pi}e^{i \lambda j}(1-e^{-i\tau\lambda})^{d}\frac{1}{(i\lambda)^{d}}d\vec{Z}_{\xi^{(d)}}(\lambda),
\]
where $\vec{Z}_{\xi^{(d)}}(\lambda)=\{Z_{k}(\lambda)\}_{k=1}^{\infty}$ is a vector-valued random process with uncorrelated increments on $[-\pi,\pi)$.

Consider the generated stationary stochastic sequence $\vec\eta_j$ defined in Section \ref{section_PC_process_i_n_c}, which is uncorrelated with the increment sequence $\vec \xi^{(d)}_j$. It admits the spectral representation
\begin{equation}
\label{SpectrPred_incr_eta_c}
\vec\eta_j=\int_{-\pi}^{\pi}e^{i\lambda j}d\vec{Z}_{\eta}(\lambda),
\end{equation}
where $\vec{Z}_{\eta}(\lambda)$   is a vector-valued random process with uncorrelated increments on $[-\pi,\pi)$.
The spectral representation  of the
sequence $\vec \zeta^{(d)}_j$, generated by the process $\zeta(t)=\xi(t)+\eta(t)$, is determined by the spectral densities $f(\lambda)$
and $g(\lambda)$ by the relation
\begin{equation}
\label{SpectrPred_incr_zeta_c}
\vec{\zeta}^{(d)}_j=\int_{-\pi}^{\pi}e^{i \lambda j}(1-e^{-i\tau\lambda})^{d}\frac{1}{(i\lambda)^{d}}d\vec{Z}_{\xi^{(d)}+\eta^{(d)}}(\lambda).
\end{equation}
The random processes $\vec{Z}_{\eta}(\lambda)$ and $\vec{Z}_{\eta^{(d)}}(\lambda)$ are connected by the relation $d\vec{Z}_{\eta^{(d)}}(\lambda)=(i\lambda)^d d\vec{Z}_{\eta }(\lambda)$,
$\lambda\in[-\pi,\pi)$, see \cite{Luz_Mokl_book}. The spectral density $p(\lambda)=\{p_{kn}(\lambda)\}_{k,n=1}^\infty$ of the
sequence $\vec \zeta^{(d)}_j$  is determined by the spectral densities $f(\lambda)$
and $g(\lambda)$ by the relation
\[p(\lambda)=f(\lambda)+\lambda^{2d}g(\lambda).\]

\section{Hilbert space projection method of interpolation}\label{classical_interpolation_i_n_c}

Consider an infinite dimension vector stationary increment sequence $\{\vec{\xi}^{(d)}_j,j\in\mr Z\}$ (\ref{generted_incr_seq_i_n_c}) generated by the increment process $\xi^{(d)}(t,\tau T)$, $t\in\mathbb R$, which has
 the spectral density matrix $f(\lambda)=\{f_{ij}(\lambda)\}_{i,j=1}^{\infty}$.

Let  $\eta(t)$, $t\in\mathbb R$, be an uncorrelated with
the process $\xi(t)$ periodically stationary stochastic process, which generates  an infinite dimension vector stationary   sequence $\{\vec{\eta}_j,j\in\mr Z\}$ (\ref{eta_cont_i_n_c}) with the spectral density matrix
$g(\lambda)=\{g_{ij}(\lambda)\}_{i,j=1}^{\infty}$.

We will assume that the mean
values of the increment sequence $\vec{\xi}^{(d)}_j$ and stationary
sequence  $\vec{\eta}_j$ equal to 0. We will also consider the increment step $\tau>0$.

By the classical \textbf{interpolation problem} we understand the problem of the mean square optimal linear  estimation  of the functional
\[ A_{NT}{\xi}=\int_{0}^{(N+1)T}a(t)\xi(t)dt\]
which depends on the unknown values of the stochastic process $\xi(t)$ with periodically correlated $d$th increments.
Estimates are based on observations of the process $\zeta(t)=\xi(t)+\eta(t)$  at points $t\in \mr R\setminus[0;(N+1)T]$, where $\eta(t)$.

Assume that spectral densities $f(\lambda)$ and $g(\lambda)$ satisfy the \emph{minimality condition}
\be
 \ip \text{Tr}\left[ \frac{\lambda^{2d}}{|1-e^{i\lambda\tau}|^{2d}}(f(\lambda)+\lambda^{2d}g(\lambda))^{-1}\right]
 d\lambda<\infty.
\label{umova11_i_n_c}
\ee
This is the necessary and sufficient condition under which the mean square errors of the optimal estimates of the functional $A_N\vec\xi$ to be defined below is not equal to $0$.

The  functional $A_{NT}{\xi}$ does not  belong to the Hilbert space   $ H=L_2(\Omega, \mathcal F, \mt P)$   of random variables  with a zero mean value and a finite variance.
To apply a Hilbert pace projection method of estimation,  represent the functional  $A_{NT}{\xi}$ as a sum of a functional with finite second moment  belonging to $H$ and a functional depended on the  observed values of the process $\zeta(t)$ (``initial values''). This representation is described by the following two lammas.

\begin{lema}[\cite{Luz_Mokl_extra_cont_PCI}]\label{nas predst A_T_cont}
The  linear functional
\[A_{NT}\zeta=\int_{0}^{(N+1)T}a(t)\zeta(t)dt\]
allows the representation
\[
    A_{NT}\zeta=B_{NT}\zeta-V_{NT}\zeta,\]
where
\[
    B_{NT}\zeta=\int_0^{(N+1)T} b^{\tau,N}(t)\zeta^{(d)}(t,\tau T)dt,\quad V_{NT}\zeta=\int_{-\tau T d}^{0}v^{\tau,N}(t)\zeta(t)dt,\]
and
\begin{eqnarray}\label{koefv_N_cont}
    v^{\tau,N}(t)&=&\sum_{l=\left \lceil-\frac{t}{\tau T}\right \rceil}^{\min\left \{\left [\frac{{(N+1)T}-t}{\tau T}\right ],d\right \}}(-1)^l{d \choose l}b^{\tau, N}(t+l\tau T),\quad t\in[-\tau T d;0),
\\
\label{koef_N b_cont}
    b^{\tau,N}(t)&=&\sum_{k=0}^{\left[\frac{{(N+1)T}-t}{\tau T}\right]}a( t +\tau T k)d(k)= D^{\tau T,N}\me a(t),\,
t\in[0;{(N+1)T}].
\end{eqnarray}
Here  $\lceil x\rceil$ denotes the least integer greater than or equal to  $x$, $[x]$ denotes the integer part of $x$, coefficients
  $\{d(k):k\geq0\}$ are determined by the relation
\[\sum_{k=0}^{\infty}d (k)x^k=\left(\sum_{j=0}^{\infty}x^{
j}\right)^d,\]
the linear transformation $D^{\tau T,N}$ acts on an arbitrary function $x(t)$, $t\in[0;{(N+1)T}]$, as follows:
\[
    D^{\tau T,N}\me x(t)=\sum_{k=0}^{\left[\frac{{(N+1)T}-t}{\tau T}\right]}x(t+\tau T k)d(k).\]
\end{lema}

From Lemma \ref{nas predst A_T_cont}, we obtain the following representation of the functional $A_{NT}\xi$:
\[A_{NT}\xi=A_{NT}\zeta-A_{NT}\eta=B_{NT}\zeta-A_{NT}\eta-V_{NT}\zeta=H_{NT}\xi-V_{NT}\zeta,\]
\[
 H_{NT}\xi:=B_{NT}\zeta-A_{NT}\eta,\]
where
\[
 A_{NT}\zeta=\int_{0}^{(N+1)T}a(t)\zeta(t)dt,\quad A_{NT}\eta=\int_{0}^{(N+1)T}a(t)\eta(t)dt,\]
\[
 B_{NT}\zeta=\int_0^{(N+1)T} b_{\tau,N}(t)\zeta^{(d)}(t,\tau T)dt,\quad
 V_{NT}\zeta=\int_{-\tau T d}^{0}v_{\tau,N}(t)\zeta(t)dt, \]
the functions $b_{\tau,N}(t)$, $t\in[0;(N+1)T]$, and $v_{\tau,N}(t)$, $t\in[-\tau
Td;0)$, are calculated by formulas (\ref{koef_N b_cont}) and
(\ref{koefv_N_cont}) respectively.

 The following lemma represent the functional $H_{NT}\xi$ in terms of the sequences   $\vec\eta_j=(\eta_{kj},k=1,2,\dots)^{\top}$ and $\vec\zeta^{(d)}_j=\vec\xi^{(d)}_j+\vec\eta^{(d)}_j=(\zeta^{(d)}_{kj},k=1,2,\dots)^{\top}$, $j\in\mr Z$.

\begin{lema}
The functional $H_{NT}\xi$ can be represented in the form
\[
 H_{NT}\xi= \sum_{j=0}^N
{(\vec{b}^{\tau,N}_j)}^{\top}\vec{\zeta}^{(d)}_j- \sum_{j=0}^N
{(\vec{a}_j)}^{\top}\vec{\eta}_j
=B_{N}\vec \zeta - V_{N}\vec \eta=:H_N\vec \xi,
\]
where
\[
\vec{b}^{\tau,N}_j=(b^{\tau,N}_{kj},k=1,2,\dots)^{\top}=
(b^{\tau,N}_{1j},b^{\tau,N}_{3j},b^{\tau,N}_{2j},\dots,b^{\tau,N}_{2k+1,j},b^{\tau,N}_{2k,j},\dots)^{\top},
\]
\[
 b^{\tau,N}_{kj}=\langle b^{\tau,N}_j,\widetilde{e}_k\rangle =
\frac{1}{\sqrt{T}} \int_0^T b^{\tau,N}_j(v)e^{-2\pi
i\{(-1)^k\left[\frac{k}{2}\right]\}v/T}dv,
\]
and
\[
\vec{a}_j=(a_{kj},k=1,2,\dots)^{\top}=
(a_{1j},a_{3j},a_{2j},\dots,a_{2k+1,j},a_{2k,j},\dots)^{\top},
\]
\[
 a_{kj}=\langle a_j,\widetilde{e}_k\rangle =
\frac{1}{\sqrt{T}} \int_0^T a_j(v)e^{-2\pi
i\{(-1)^k\left[\frac{k}{2}\right]\}v/T}dv,\quad k=1,2,\dots,\,j=0,1,\ldots,N.
\]
The coefficients $\{\vec{a}_j, j=0,1,\dots,N\}$ and  $\{\vec{b}^{\tau,N}_j, j=0,1,\dots,N\}$ are related as
\be \label{a_b_N_relation}
\vec{b}_j^{\tau,N}=\sum_{m=j}^{N}\mt{diag}_{\infty}(d_{\tau}(m-j))\vec{a}_m=(D^{\tau}_N{\me a_N})_j,  \quad j=0,1,\dots,N.
\ee
where
$\me a_N=((\vec{a}_0)^{\top},(\vec{a}_1)^{\top}, \ldots,(\vec{a}_N)^{\top})^{\top}$, the coefficients $\{d_{\tau}(k):k\geq0\}$ are determined by the relationship
\[
 \sum_{k=0}^{\infty}d_{\tau}(k)x^k=\left(\sum_{j=0}^{\infty}x^{\tau j}\right)^d,\]
$D^{\tau}_N$ is a linear transformation  determined by a $N\times N$ matrix with the infinite dimension matrix entries
 $D^{\tau}_N(k,j), k,j=0,1,\dots,N$ such that $D^{\tau}_N(k,j)=\mt{diag}_{\infty}(d_{\tau}(j-k))$ if $0\leq k\leq j \leq N$ and $D^{\tau}_N(k,j)=\mt{diag}_{\infty}(0)$ for $0\leq j<k\leq N$; $\mt{diag}_{\infty}(x)$ denotes an infinite dimensional diagonal matrix with the entry $x$ on its diagonal.
\end{lema}
\begin{proof}
Define
\[
b^{\tau,N}_j(u)=b^{\tau,N}(u+jT),\quad \zeta^{(d)}_j(u)=\zeta^{(d)}_j(u+jT,\tau T),\,\, u\in [0,T),
\]
and
\[
a_j(u)=a(u+jT),\quad \eta_j(u)=\eta_j(u+jT),\,\, u\in [0,T).
\]

Making use of the decomposition (\ref{zeta_cont_i_n_c}) for the generated stationary increment sequence  $\{\zeta^{(d)}_j,j\in\mathbb Z\}$  and the solution of equation
\begin{equation} \label{peretv}
(-1)^k\left[\frac{k}{2}\right]+(-1)^m\left[\frac{m}{2}\right]=0
\end{equation}
of two variables  $(k,m)$, which is given by pairs  $(1,1)$, $(2l+1,2l)$ and
$(2l,2l+1)$ for $l=2,3,\dots$, the functional $B_{NT}\zeta$ can be rewritten in the form
\cite{MoklyachukGolichenko2016}
\begin{eqnarray*}
    B_{NT}\zeta&=&\int_0^{(N+1)T} b^{\tau,N}(t)\zeta^{(d)}(t,\tau T)dt
= \sum_{j=0}^{N}\int_{0}^{T}
b^{\tau,N}_j(u)\zeta^{(d)}_j(u)
du
\\
&=&\sum_{j=0}^N\frac{1}{T}\int_{0}^{T}\left(\sum_{k=1}^\infty b^{\tau,N}_{kj}
e^{2\pi i\{(-1)^k\left[\frac{k}{2}\right]\}u/T} \right)
\left(\sum_{m=1}^\infty \zeta^{(d)}_{mj} e^{2\pi
i\{(-1)^m\left[\frac{m}{2}\right]\}u/T} \right)du
\\
&=&\sum_{j=0}^N\sum_{k=1}^\infty \sum_{m=1}^\infty b^{\tau,N}_{kj} \zeta^{(d)}_{mj}
\frac{1}{T}\int_{0}^{T}e^{2\pi
i\left\{(-1)^k\left[\frac{k}{2}\right]+(-1)^m\left[\frac{m}{2}\right]\right\}u/T}
du
\\
&=&\sum_{j=0}^N\sum_{k=1}^{\infty} b^{\tau,N}_{kj}\zeta^{(d)}_{kj}=
 \sum_{j=0}^N
{(\vec{b}^{\tau,N}_j)}^{\top}\vec{\zeta}^{(d)}_j
\\
&=&B_{N}\vec \zeta.
\end{eqnarray*}

In the same way we obtain the representation of the the functional $V_{NT}\eta$
\begin{eqnarray*}
    V_{NT}\eta&=&\int_0^{(N+1)T} a(t)\eta(t)dt
= \sum_{j=0}^{N}\int_{0}^{T}
a_j(u)\eta_j(u)
du
\\
&=&\sum_{j=0}^N\sum_{k=1}^{\infty} a_{kj}\eta_{kj}=
 \sum_{j=0}^N
{(\vec{a}_j)}^{\top}\vec{\eta}_j=V_{N}\vec \eta.
\end{eqnarray*}

Then from Lemma \ref{nas predst A_T_cont} obtain
\[
b^{\tau,N}_j(u)=\sum_{l=0}^{\left[\frac{{(N+1)T}-u-jT}{\tau T}\right]}a( u+jT +\tau T l)d(l)= D^{\tau T,N} a(u),\,
u\in[0;(N+1)T), \, j=0,1,\ldots,N,
\]
and
\[
b^{\tau,N}_{kj}=\sum_{l=0}^{\left[\frac{N-j}{\tau}\right]}a_{kj+\tau l}d(l),\,
 j=0,1,\ldots,N,
\]
which finalizes the proof.
\end{proof}

Assume, that the coefficients $\{\vec{a}_j, j=0,1,\dots,N\}$ and  $\{\vec{b}^{\tau,N}_j, j=0,1,\dots,N\}$  satisfy the conditions
\begin{equation} \label{coeff_a_i_n_c}
 \|\vec{a}_j\|<\infty, \quad \|\vec{a}_j\|^2=\sum_{k=1}^\infty |{a}_{kj}|^2,\quad    j=0,1,\dots,N.
\end{equation}
and
\begin{equation} \label{coeff_b_i_n_c}
 \|\vec{b}^{\tau,N}_j\|<\infty, \quad \|\vec{b}^{\tau,N}_j\|^2=\sum_{k=1}^\infty |{b}^{\tau,N}_{kj}|^2,\quad    j=0,1,\dots,N.
\end{equation}
It follows from   conditions (\ref{coeff_a_i_n_c}) - (\ref{coeff_b_i_n_c}) that the functional
$
 H_{NT}\xi = H_{N}\vec \xi
$
has a finite second moment and the functional  $V_{NT}\zeta$ depends on the  observations of the   process $ \xi(t)+\eta(t)$ at points $t\in[\tau T d;0)$.
Therefore, the estimates $\widehat{A}_N\vec\xi$ and $\widehat{H}_N\vec\xi$ of the functionals $A_N\vec\xi$ and $H_N\vec\xi$, as well as  the mean-square errors $\Delta(f,g;\widehat{A}_{NT}\xi)=\Delta(f,g;\widehat{A}_N\vec\xi)=\mt E |A_N\vec\xi-\widehat{A}_N\vec\xi|^2$ and $\Delta(f,g;\widehat{H}_{NT}\xi)=\Delta(f,g;\widehat{H}_N\vec\xi)=\mt E
|H_N\vec\xi-\widehat{H}_N\vec\xi|^2$ of the estimates $\widehat{A}_N\vec\xi$ and $\widehat{H}_N\vec\xi$ satisfy the following relations
\be\label{mainformula_i_n_c}
    \widehat{A}_N\vec\xi=\widehat{H}_N\vec\xi-V_{NT}\zeta,\ee
\[
    \Delta(f,g;\widehat{A}_N\vec\xi)
    =\mt E |A_N\vec\xi-\widehat{A}_N\vec\xi|^2
    =\mt E|H_N\vec\xi-V_{NT}\zeta-\widehat{H}_N\vec\xi+V_{NT}\zeta|^2
    \]
    \[
    =\mt E|H_N\vec\xi-\widehat{H}_N\vec\xi|^2=\Delta(f,g;\widehat{H}_N\vec\xi).
    \]
Thus, it is enough to find the estimate of the  functional $H_N\vec\xi$.

Making use of   spectral representations (\ref{SpectrPred_incr_zeta_c}) and (\ref{SpectrPred_incr_eta_c}), onr can  obtain the following spectral representation of the  functional $H_N\vec\xi$
\[H_N\vec\xi=\int_{-\pi}^{\pi}\left(\vec{B}_{\tau,N}(e^{i\lambda})\right)^{\top}
\frac{(1-e^{-i\lambda \tau})^d}{(i\lambda)^d}d\vec{Z}_{\xi^{(d)}+\eta^{(d)}}(\lambda)
-\int_{-\pi}^{\pi}\left(\vec{A}_N(e^{i\lambda})\right)^{\top}d\vec{Z}_{\eta}(\lambda),\]
where
\[\vec{B}_{\tau,N}(e^{i\lambda})=\sum_{j=0}^{N}\vec{b}_j^{\tau,N}e^{i\lambda j}=\sum_{j=0}^{N}(D^{\tau}_N\me a_N)_je^{i\lambda j},\quad \vec{A}_N(e^{i\lambda})=\sum_{j=0}^{N}\vec{a}_je^{i\lambda j}.\]

The next step is to define the subspace of $H=L_2(\Omega,\mathcal{F},\mt P)$ generated by the observations. Consider   the closed linear subspaces
\[H^{0-}(\xi^{(d)}_{\tau}+\eta^{(d)}_{\tau})
=\overline{span}\{\xi^{(d)}_k(j,\tau)+\eta^{(d)}_k(j,\tau):k=1,\dots,\infty;j=-1,-2,-3,\dots\}\]
and
\[H^{N+}(\xi^{(d)}_{-\tau}+\eta^{(d)}_{-\tau})=\overline{span}
\{\xi^{(d)}_k(j,-\tau)+\eta^{(d)}_k(j,-\tau):k=1,\dots,\infty;j=N+1,N+2,\dots\}.\]
The equality $\xi^{(d)}_k(j,-\tau)=(-1)^d\xi^{(d)}_k(j+\tau d,\tau)$ implies
\[
    H^{N+}(\xi_{-\tau}^{(d)}+\eta_{-\tau}^{(d)})=H^{(N+\tau d)+}(\xi_{\tau}^{(d)}+\eta_{\tau}^{(d)}).\]

Define the following   closed linear subspaces of the Hilbert space
$L_2(f(\lambda)+\lambda^{2d}g(\lambda))$  of vector-valued functions endowed by the inner product $\langle g_1;g_2\rangle=\ip (g_1(\lambda))^{\top}(f(\lambda)+\lambda^{2d}g(\lambda))\overline{g_2(\lambda)}d\lambda$ which is generated by the functions
\begin{multline*}
L_2^{0-}(f(\lambda)+\lambda^{2d}g(\lambda))
\\
=\{
 e^{i\lambda j}(1-e^{-i\lambda \tau})^d\frac{1}{(i\lambda)^d}{\delta}_{k},    \, k=1,\dots,\infty;\,\, j=-1,-2,-3,\dots\}
\end{multline*}
and
\begin{multline*}
L_2^{N+}(f(\lambda)+\lambda^{2d}g(\lambda))
\\
=\{
 e^{i\lambda j}(1-e^{-i\lambda \tau})^d\frac{1}{(i\lambda)^d}{\delta}_{k},  \, k=1,\dots,\infty;\,\, \,j=N+1,N+2,\dots\},
\end{multline*}
where ${\delta}_k=\{\delta_{kl}\}_{l=1}^{\infty}$, $\delta_{kl}$ are Kronecker symbols.

Representation
 (\ref{SpectrPred_incr_zeta_c})
yields a  map between the elements $e^{i\lambda j}(1-e^{-i\lambda\tau})^d(i\lambda)^{-d}$ of the space
$$L_2^{0-}(f(\lambda)+\lambda^{2d}g(\lambda))\oplus L_2^{N+}(f(\lambda)+\lambda^{2d}g(\lambda))$$
and the elements
$\vec\xi^{(d)}_j+\vec\eta^{(d)}_j$
of the space
$$H^{0-}(\xi^{(d)}_{\tau}+\eta^{(d)}_{\tau})\oplus H^{N+}(\xi^{(d)}_{-\tau}+\eta^{(d)}_{-\tau})=
H^{0-}(\xi^{(d)}_{\tau}+\eta^{(d)}_{\tau})\oplus H^{(N+\tau d)+}(\xi^{(d)}_{\tau}+\eta^{(d)}_{\tau}).$$

Relation (\ref{mainformula_i_n_c}) implies that a linear estimate $\widehat{A}_N\vec\xi$ of  $A_N\vec\xi$
can be represented as
\be \label{otsinka_A_N_i_n_c}
 \widehat{A}_N\vec\xi=\ip
(\vec{h}_{\tau,N}(\lambda))^{\top}d\vec{Z}_{\xi^{(d)}+\eta^{(d)}}(\lambda)-\int_{-\tau T d}^{0}v_{\tau,N}(t)\zeta(t)dt,
\ee
 where
$\vec{h}_{\tau,N}(\lambda)=\{h_{k}^{\tau,N}(\lambda)\}_{k=1}^{\infty}$ is the spectral characteristic of the   estimate $\widehat{H}_N\vec\xi$.

Now we can apply  the Hilbert space projection method proposed by Kolmogorov \cite{Kolmogorov}.
The mean square optimal estimate
$\widehat{H}_N\vec\xi$ is found  as a projection of the element $H_N\vec\xi$ on the
subspace $H^{0-}(\xi^{(d)}_{\tau}+\eta^{(d)}_{\tau})\oplus H^{(N+\tau d)+}(\xi^{(d)}_{\tau}+\eta^{(d)}_{\tau})$. This projection is
characterized by two conditions:

1) $ \widehat{H}_N\vec\xi\in H^{0-}(\xi^{(d)}_{\tau}+\eta^{(d)}_{\tau})\oplus H^{(N+\tau d)+}(\xi^{(d)}_{\tau}+\eta^{(d)}_{\tau}) $;

2) $(H_N\vec\xi-\widehat{H}_N\vec\xi)
\perp
H^{0-}(\xi^{(d)}_{\tau}+\eta^{(d)}_{\tau})\oplus H^{(N+\tau d)+}(\xi^{(d)}_{\tau}+\eta^{(d)}_{\tau})$.

The following relation follows from condition 2) and   holds true for all $j\leq-1$ and $j\geq N+\tau d+1$
\begin{multline}
\int_{-\pi}^{\pi}
\ld[\ld(\vec{B}_{\tau,N}(e^{i\lambda})\frac{(1-e^{-i\lambda\tau})^d}{(i\lambda)^d}-
\vec{h}_{\tau,N}(\lambda)\rd)^{\top}
p(\lambda)-
(\vec{A}_N(e^{i\lambda}))^{\top}g(\lambda)(-i\lambda)^d\rd]\times
\\
\times
\frac{(1-e^{i\lambda\tau})^d}{(-i\lambda)^{d}}e^{-i\lambda j}d\lambda=0.
\end{multline}
Thus, the spectral characteristic $\vec{h}_{\tau,N}(\lambda)$ of the estimate
$\widehat{H}_N\vec\xi$ can be represented as
\begin{multline}
(\vec{h}_{\tau,N}(\lambda))^{\top}=
(\vec{B}_{\tau,N}(e^{i\lambda}))^{\top}
\frac{(1-e^{-i\lambda \tau})^d}{(i\lambda)^d}
\\
-
\ld((\vec{A}_{\tau,N}(e^{i\lambda}))^{\top}g(\lambda)+(\vec{C}_{\tau,N}(e^{i\lambda}))^{\top}\rd)\frac{(-i\lambda)^{d}}{(1-e^{i\lambda \tau})^d}
(f(\lambda)+\lambda^{2d}g(\lambda))^{-1},
 \label{spectr_A_N_i_n_c}
\end{multline}
where
\[
\vec{A}_{\tau,N}(e^{i\lambda})=(1-e^{i\lambda \tau})^d\vec{A}_{N}(e^{i\lambda})=\sum_{j=0}^{N+\tau d}\vec{a}_j^{\tau,N}e^{i\lambda j},\quad \vec{C}_{\tau,N}(e^{i \lambda})=\sum_{j=0}^{N+\tau d}\vec{c}_j^{\tau,N}e^{i\lambda j},
\]
the coefficients $\{\vec{a}_j^{\tau,N}:0\leq j\leq N+\tau d\}$ are calculated by the formula
\be\label{coeff a_N_mu_i_n_c}
    \vec{a}_j^{\tau,N}
    =\sum_{l=\max\ld\{\ld\lceil\frac{j-N}{\tau}\rd\rceil,0\rd\}}^{\min\ld\{\ld[\frac{j}{\tau}\rd],d\rd\}}
    (-1)^l{d \choose l}\vec{a}(j-\tau l),\quad 0\leq j\leq N+\tau d,\ee
the coefficients  $\vec{c}_j^{\tau,N}=\{c_{kj}^{\tau,N}\}_{k=1}^{\infty}, j=0,1,\dots,N+\tau d,$  are unknown coefficients to be found.

Condition 1) implies that the following equations should be satisfied
\[
\int_{-\pi}^{\pi} \biggl[(\vec{B}_{\tau,N}(e^{i\lambda}))^{\top}-
(\vec{A}_{\tau,N}(e^{i\lambda}))^{\top}g(\lambda)
\frac{\lambda^{2d}}{(1-e^{-i\lambda \tau})^d(1-e^{i\lambda \tau})^d}
(f(\lambda)+\lambda^{2d}g(\lambda))^{-1}
-
\]
\begin{equation} \label{eq_C_N_i_n_c}
-
\frac{\lambda^{2d}(\vec{C}_{\tau,N}(e^{i\lambda}))^{\top}}{(1-e^{-i\lambda \tau})^d (1-e^{i\lambda \tau})^d}(f(\lambda)+\lambda^{2d}g(\lambda))^{-1}\biggr]e^{-i\lambda j}d\lambda=0,\,  0\leq j\leq N+\tau d.
\end{equation}

 Define for $0\leq l,j \leq N+\tau d$ the Fourier coefficients of the corresponding functions
\begin{eqnarray*}
T^{\tau}_{l,j}&=&\frac{1}{2\pi}\int_{-\pi}^{\pi}
e^{i\lambda(j-l)}\frac{\lambda^{2d}}{|1-e^{i\lambda\tau}|^{2d}}
\ld[g(\lambda)(f(\lambda)+\lambda^{2d}g(\lambda))^{-1}\rd]^{\top}
d\lambda;
\\
P_{l,j}^{\tau}&=&\frac{1}{2\pi}\int_{-\pi}^{\pi} e^{i\lambda (j-l)}
\dfrac{\lambda^{2d}}{|1-e^{i\lambda\tau}|^{2d}}\ld[(f(\lambda)+\lambda^{2d}g(\lambda))^{-1}\rd]^{\top}
d\lambda;
\\
 Q_{l,j}&=&\frac{1}{2\pi}\int_{-\pi}^{\pi}
e^{i\lambda(j-l)}\ld[f(\lambda)(f(\lambda)+\lambda^{2d}g(\lambda))^{-1}g(\lambda)\rd]^{\top}d\lambda.
\end{eqnarray*}

  \noindent Making use of the defined Fourier coefficients, relation \eqref{eq_C_N_i_n_c} can be presented as a system of $N+\tau d+1$ linear equations
  determining the unknown coefficients $\vec{c}_j^{\tau,N}$, $0\leq j\leq N+\tau d$.
\begin{eqnarray*}\label{linear equations1_i_n_c}
    \vec{b}_l^{\tau,N}-\sum_{ j=0}^{N+\tau d}T^{\tau}_{l,j}\vec{a}_j^{\tau,N}
    &=&\sum_{j=0}^{N+\tau d}P_{l,j}^{\tau}\vec{c}_j^{\tau,N},\quad 0\leq l\leq N,
\\
\label{linear equations2_i_n_c}
    -\sum_{ j=0}^{N+\tau d}T^{\tau}_{l,j}\vec{a}_j^{\tau,N}
    &=&\sum_{j=0}^{N+\tau d}P_{l,j}^{\tau}\vec{c}_j^{\tau,N},\quad N+1\leq l\leq N+\tau d.
\end{eqnarray*}

  \noindent Denote by $[D_N^{\tau}\me a_N]_{+\tau d}$ a vector of dimension  $N+\tau d +1$ with infinite dimensional entries, which is constructed by adding
$\tau d$ infinite dimensional zero vectors  to the vector $D_N^{\tau}\me a_N$ of dimension  $N+1$.
Then  the system (\ref{linear equations1_i_n_c}) -- (\ref{linear equations2_i_n_c}) can be represented in the matrix form
\[
    [D_N^{\tau}\me a_N]_{+\tau d}-\me T^{\tau}_N\me a^{\tau}_N=\me P^{\tau}_N\me c^{\tau}_N,\]
where
\begin{eqnarray*}
    \me a^{\tau}_N&=&((\vec{a}_0^{\tau,N})^{\top},(\vec{a}_1^{\tau,N})^{\top},(\vec{a}_2^{\tau,N})^{\top},\ldots,
(\vec{a}_{N+\tau d}^{\tau,N})^{\top})^{\top}
\\
    \me c^{\tau}_N&=&((\vec{c}_0^{\tau,N})^{\top},(\vec{c}_1^{\tau,N})^{\top},(\vec{c}_2^{\tau,N})^{\top},\ldots,
(\vec{c}_{N+\tau d}^{\tau,N})^{\top})^{\top}
\end{eqnarray*}
are  vectors of  dimension $N+\tau d+1$ with the infinite dimensional entries; $\me P^{\tau}_N$ and $\me T^{\tau}_N$ are matrices of dimension $(N+\tau d+1)\times (N+\tau d+1)$ with the infinite dimensional matrix entries $(\me P^{\tau}_N)_{l,j}=P_{j,l}^{\tau}$ and $(\me T^{\tau}_N)_{l,j} =T^{\tau}_{l,j}$, $0\leq l,j\leq N+\tau d$.

Thus, the coefficients $\vec{c}_j^{\tau,N}$, $0\leq j\leq N+\tau d$, are determined by the formula
\[
  \vec{c}_j^{\tau,N}=\ld((\me P^{\tau}_N)^{-1}[D_N^{\tau}\me a_N]_{+\tau d}-(\me P^{\tau}_N)^{-1}\me T^{\tau}_N\me a^{\tau}_{N}\rd)_j,\quad 0\leq j\leq N+\tau d,\]
where $\ld((\me P^{\tau}_N)^{-1}[D_N^{\tau}\me a_N]_{+\tau d}-(\me P^{\tau}_N)^{-1}\me T^{\tau}_N\me a^{\tau}_N\rd)_j$, $0\leq j\leq N+\tau d$, is the
 $j$th element of the vector  $(\me P^{\tau}_N)^{-1}[D_N^{\tau}\me a_N]_{+\tau d}-(\me P^{\tau}_N)^{-1}\me T^{\tau}_N\me a^{\tau}_N$.

The existence of the inverse matrix $(\me P^{\tau}_N)^{-1}$ was justified  in \cite{Luz_Mokl_book} under condition (\ref{umova11_i_n_c}).

The spectral characteristic  $\vec{h}_{\tau,N}(\lambda)$ of the estimate $\widehat{H}_N\xi$ of the functional $H_N\xi$ is calculated by formula (\ref{spectr_A_N_i_n_c}), where
\be\label{spectr_C_N_i_n_c}
 \vec{C}_{\tau,N}(e^{i \lambda})=\sum_{j=0}^{N+\tau d}
    \ld((\me P^{\tau}_N)^{-1}[D_N^{\tau}\me a_N]_{+\tau d}-(\me P^{\tau}_N)^{-1}\me T^{\tau}_N\me a^{\tau}_N\rd)_j e^{i\lambda j}.
\ee

Denote
\begin{eqnarray*}
\me C^{f}_{\tau,N}(e^{i\lambda})
&=&
\overline{g(\lambda)}\vec{A}_{\tau,N}(e^{i\lambda}) +
\vec{C}_{\tau,N}(e^{i \lambda}),
\\
{\me C}^{g}_{\tau,N}(e^{i \lambda})
&=&
{|1-e^{i\lambda \tau}|^{2n}}\lambda^{-2n}\overline{f^0(\lambda)}\vec{A}_N(e^{i\lambda})
-{(1-e^{-i\lambda \tau})^n}
\vec{C}_{\tau,N}(e^{i \lambda}).
\end{eqnarray*}

The value of the  mean-square errors of the estimates $\widehat{A}_{NT}\xi$ and $\widehat{H}_N\vec\xi$ can be calculated by the formula
\begin{eqnarray}
\notag
\Delta(f,g;\widehat{A}_{NT}\xi)&=&\Delta(f,g;\widehat{H}_N\vec\xi)= \mt E|H_N\vec\xi-\widehat{H}_N\vec\xi|^2
\\\notag&=&
\frac{1}{2\pi}\int_{-\pi}^{\pi}
\frac{\lambda^{2d}}{|1-e^{i\lambda\tau}|^{2d}}
(\me C^{f}_{\tau,N}(e^{i\lambda}))^{\top}(f(\lambda)+\lambda^{2d}g(\lambda))^{-1}\, f(\lambda)\,
\\\notag
&&\quad
\times
(f(\lambda)+\lambda^{2d}g(\lambda))^{-1}\overline{\me C^{f}_{\tau,N}(e^{i\lambda})}
d\lambda
\\\notag
&&
+\frac{1}{2\pi}\int_{-\pi}^{\pi}
\frac{\lambda^{4d}}{|1-e^{i\lambda\tau}|^{4d}}
({\me C}^{g}_{\tau,N}(e^{i \lambda}))^{\top}(f(\lambda)+\lambda^{2d}g(\lambda))^{-1}\,g(\lambda)\,
\\\notag
&&\quad
\times
(f(\lambda)+\lambda^{2d}g(\lambda))^{-1}\overline{{\me C}^{g}_{\tau,N}(e^{i \lambda})}
d\lambda
\\\notag
&=& \ld\langle [D_N^{\tau}\me a_N]_{+\tau d}- \me T^{\tau}_N\me
    a^{\tau}_N,(\me P^{\tau}_N)^{-1}[D_N^{\tau}\me a_N]_{+\tau d}
    -(\me P^{\tau}_N)^{-1}\me T^{\tau}_N\me a^{\tau}_N\rd\rangle
    \\
&&+\langle\me Q_N\me a_N,\me
    a_N\rangle,
    \label{pohybka_A_N_i_n_c}
\end{eqnarray}
where $\me Q_N$ is a matrix of the dimension $(N+1)\times(N+1)$ with the infinite dimensional matrix entries $(\me Q_N)_{j,l}=Q_{j,l}$, $0\leq j,l\leq N$. Thus, the following theorem holds true.

\begin{thm}\label{thm_A_N_i_n_c}
Consider uncorrelated  a stochastic
process $\xi(t)$, $t\in \mr R$ with a periodically stationary  increments and   a periodically stationary stochastic
process $\eta(t)$, $t\in \mr R$, which determine a
generated stationary  $d$th increment sequence
$\vec{\xi}^{(d)}_j$ with the spectral density matrix $f(\lambda)=\{f_{kn}(\lambda)\}_{k,n=1}^{\infty}$ and
   a
generated stationary  sequence
$\vec{\eta}_j$ with the spectral density matrix $g(\lambda)=\{g_{kn}(\lambda)\}_{k,n=1}^{\infty}$ respectively.
Let the coefficients $\vec {a}_j$, $\vec{b}^{\tau,N}_j$, $j=0,1,\ldots,N$ generated by the function $a(t)$, $t\geq 0$, satisfy conditions  (\ref{coeff_a_i_n_c}) - (\ref{coeff_b_i_n_c}).
 Let  minimality condition
(\ref{umova11_i_n_c}) be satisfied. The
optimal linear estimate $\widehat{A}_{NT}\xi$ of the functional
$A_{NT}\xi$  based on
observations of the process $\xi(t)+\eta(t)$ at points   $t\in \mr R\setminus[0;(N+1)T]$
is calculated by  formula (\ref{otsinka_A_N_i_n_c}).
The spectral characteristic $\vec{h}_{\tau,N}(\lambda)$ of the optimal
estimate $\widehat{A}_{NT}\xi$ and the mean-square error
$\Delta(f,g;\widehat{A}_{NT}\xi)$ are  calculated by  formulas (\ref{spectr_A_N_i_n_c}) -- (\ref{spectr_C_N_i_n_c}) and (\ref{pohybka_A_N_i_n_c}) respetively.
\end{thm}

\begin{nas}
The spectral characteristic $\vec{h}_{\tau,N}(\lambda)$ (\ref{spectr_A_N_i_n_c}) admits the representation
$\vec{h}_{\tau,N}(\lambda)=\vec{h}_{\tau,N}^{1}(\lambda)-\vec{h}_{\tau,N}^2(\lambda)$, where
\begin{multline*}
(\vec{h}_{\tau}^1(\lambda))^{\top}=(\vec{B}_{\tau,N}(e^{i\lambda}))^{\top}
\frac{(1-e^{-i\lambda \tau})^d}{(i\lambda)^d}-
\\
-\frac{(-i\lambda)^{d}}{(1-e^{i\lambda \tau})^d}
\left(
\sum_{j=0}^{N+\tau d}
    \ld((\me P^{\tau}_N)^{-1}[D_N^{\tau}\me a_N]_{+\tau d}\rd)_j e^{i\lambda j}
\right)^{\top}
(f(\lambda)+\lambda^{2d}g(\lambda))^{-1},
\end{multline*}
\begin{multline*}
(\vec{h}_{\tau,N}^2(\lambda))^{\top}=
(\vec{A}_N(e^{i\lambda }))^{\top}
 {(-i\lambda)^dg(\lambda)}(f(\lambda)+\lambda^{2d}g(\lambda))^{-1}-
 \\
-\frac{(-i\lambda)^{d}}{(1-e^{i\lambda \tau})^d}
\left(
\sum_{j=0}^{N+\tau d}
    \ld((\me P^{\tau}_N)^{-1}\me T^{\tau}_N\me a^{\tau}_N\rd)_j e^{i\lambda j}
\right)^{\top}
(f(\lambda)+\lambda^{2d}g(\lambda))^{-1}.
\end{multline*}

Here $\vec{h}_{\tau,N}^1(\lambda)$ and $\vec{h}_{\tau,N}^2(\lambda)$ are spectral characteristics of the optimal estimates
$\widehat{B}_N\vec\zeta$ and $\widehat{A}_N\vec\eta$ of the functionals $B_N\vec\zeta$ and $A_N\vec\eta$ respectively based on observations
$\xi(t)+\eta(t)$ at points of the set $t\in \mr R\setminus[0;(N+1)T]$.
\end{nas}

\section{Minimax-robust method of estimation}\label{minimax_estimation_i_n_c}

The values of the mean square errors and the spectral characteristics of the optimal estimate
of the functional ${A}_{NT}\xi$
depending on the unobserved values of a stochastic process ${\xi}(t)$ which determine a generated stationary stochastic $d$th increment sequence
$\vec{\xi}^{(d)}_j$ with the spectral density matrix $f(\lambda)$
based on observations of the process
$\xi(t)+\eta(t)$ at points $\mr R\setminus[0;(N+1)T]$ can be calculated by formulas
(\ref{spectr_A_N_i_n_c}), (\ref{spectr_C_N_i_n_c}), \eqref{pohybka_A_N_i_n_c}
respectively, under the condition that
spectral densities
$f(\lambda)$ and $g(\lambda)$ of stochastic sequences $\vec\xi_j$ and
$\vec\eta_j$ are exactly known.

In practical cases, however, spectral densities  usually are not exactly known.
If in such cases a set $\md D=\md D_f\times\md D_g$ of admissible spectral densities is defined,
the minimax-robust approach to
estimation of linear functionals depending on unobserved values of stochastic sequences with stationary increments may be applied.
This method consists in finding an estimate that minimizes
the maximal values of the mean square errors for all spectral densities
from a given class $\md D=\md D_f\times\md D_g$ of admissible spectral densities
simultaneously.

To formalize this approach we present the following definitions.

\begin{ozn}
For a given class of spectral densities $\mathcal{D}=\md
D_f\times\md D_g$ the spectral densities
$f_0(\lambda)\in\mathcal{D}_f$, $g_0(\lambda)\in\md D_g$ are called
least favorable in the class $\mathcal{D}$ for the optimal linear
estimation of the functional $A_N\vec \xi$  if the following relation holds
true:
\[\Delta(f_0,g_0)=\Delta(h(f_0,g_0);f_0,g_0)=
\max_{(f,g)\in\mathcal{D}_f\times\md
D_g}\Delta(h(f,g);f,g).\]
\end{ozn}

\begin{ozn}
For a given class of spectral densities $\mathcal{D}=\md
D_f\times\md D_g$ the spectral characteristic $h^0(\lambda)$ of
the optimal linear estimate of the functional $A_N\vec \xi$ is called
minimax-robust if there are satisfied the conditions
\[h^0(\lambda)\in H_{\mathcal{D}}=\bigcap_{(f,g)\in\mathcal{D}_f\times\md D_g}L_2^{0-}(f(\lambda)+\lambda^{2d}g(\lambda))\]
and
\[\min_{h\in H_{\mathcal{D}}}\max_{(f,g)\in \mathcal{D}_f\times\md D_g}\Delta(h;f,g)=\max_{(f,g)\in\mathcal{D}_f\times\md
D_g}\Delta(h^0;f,g).\]
\end{ozn}

Taking into account the introduced definitions and the derived relations we can verify that the following lemma holds true.

\begin{lema}
The spectral densities $f^0\in\mathcal{D}_f$,
$g^0\in\mathcal{D}_g$ which satisfy   minimality condition (\ref{umova11_i_n_c})
are least favorable in the class $\md D=\md D_f\times\md D_g$ for
the optimal linear estimation of the functional $A_N\vec\xi$ based on observations of the process $\xi(t)+\eta(t)$
at points   $t\in\mr R\setminus[0;(N+1)T]$
if the matrices $ (\me P^{\tau}_N)^0$, $(\me T^{\tau}_N)^0$, $(\me Q_N)^0$ whose entries are defined by the Fourier coefficients of the functions
\[
\frac{\lambda^{2d}}{|1-e^{i\lambda\tau}|^{2d}}\ld[g^0(\lambda)
(f^0(\lambda)+\lambda^{2d}g^0(\lambda))^{-1}\rd]^{\top},\quad
\dfrac{\lambda^{2d}}{|1-e^{i\lambda\tau}|^{2d}}\ld[(f^0(\lambda)+\lambda^{2d}g^0(\lambda))^{-1}\rd]^{\top},
\]
\[
\ld[f^0(\lambda)(f^0(\lambda)+\lambda^{2d}g^0(\lambda))^{-1}g^0(\lambda)\rd]^{\top}
\]
determine a solution of the constraint optimisation problem

\begin{multline*}
   \max_{(f,g)\in \mathcal{D}_f\times\md D_g}
   \ld(\ld\langle [D_N^{\tau}\me a_N]_{+\tau d}- \me T^{\tau}_N\me
    a_{\tau},(\me P^{\tau}_N)^{-1}[D_N^{\tau}\me a_N]_{+\tau d}-
       (\me P^{\tau}_N)^{-1}\me T^{\tau}_N\me a^{\tau}_N\rd\rangle\right.
       \\
      +\left.\langle\me Q_N\me
    a_N,\me a_N\rangle\rd)=
\end{multline*}
\begin{multline}
    = \ld\langle [D_N^{\tau}\me a_N]_{+\tau d}- (\me T^{\tau}_N)^0\me
    a^{\tau}_N,((\me P^{\tau}_N)^0)^{-1}[D_N^{\tau}\me a_N]_{+\tau d}-
 ((\me P^{\tau}_N)^0)^{-1}(\me T^{\tau}_N)^0\me a^{\tau}_N
    \rd\rangle
    \\
    +
        \langle\me Q^0_N\me
    a_N,\me a_N\rangle.
\label{minimax1_A_N_i_n_c}
\end{multline}

The minimax spectral characteristic $h^0=h_{\tau,N}(f^0,g^0)$ is calculated by formula (\ref{spectr_A_N_i_n_c}) if
$h_{\tau,N}(f^0,g^0)\in H_{\mathcal{D}}$.
\end{lema}

For more detailed analysis of properties of the least favorable spectral densities and minimax-robust spectral characteristics we observe that
the minimax spectral characteristic $h^0$ and the least favourable spectral densities $(f^0,g^0)$
form a saddle
point of the function $\Delta(h;f,g)$ on the set
$H_{\mathcal{D}}\times\mathcal{D}$.

The saddle point inequalities
\[\Delta(h;f^0,g^0)\geq\Delta(h^0;f^0,g^0)\geq\Delta(h^0;f,g)
\quad\forall f\in \mathcal{D}_f,\forall g\in \mathcal{D}_g,\forall
h\in H_{\mathcal{D}}\] hold true if $h^0=h_{\tau,N}(f^0,g^0)$ and
$h_{\tau,N}(f^0,g^0)\in H_{\mathcal{D}}$, where $(f^0,g^0)$  is a
solution of the  constraint optimization problem
\be  \label{cond-extr1_A_N_i_n_c}
\widetilde{\Delta}(f,g)=-\Delta(h_{\tau,N}(f^0,g^0);f,g)\to
\inf,\quad (f,g)\in \mathcal{D},\ee
where the functional $\Delta(h_{\tau,N}(f^0,g^0);f,g)$ is calculated by the formula
\begin{eqnarray*}
\notag
&& \Delta(h_{\tau,N}(f^0,g^0);f,g)=
\\
&=&
\frac{1}{2\pi}\int_{-\pi}^{\pi}
\frac{\lambda^{2d}}{|1-e^{i\lambda\tau}|^{2d}}
(\me C^{f0}_{\tau,N}(e^{i\lambda}))^{\top}(f^0(\lambda)+\lambda^{2d}g^0(\lambda))^{-1}\, f(\lambda)\,
\\\notag
&&\quad
\times
(f^0(\lambda)+\lambda^{2d}g^0(\lambda))^{-1}\overline{\me C^{f0}_{\tau,N}(e^{i\lambda})}
d\lambda
\\\notag
&&
+\frac{1}{2\pi}\int_{-\pi}^{\pi}
\frac{\lambda^{4d}}{|1-e^{i\lambda\tau}|^{4d}}
({\me C}^{g0}_{\tau,N}(e^{i \lambda}))^{\top}(f^0(\lambda)+\lambda^{2d}g^0(\lambda))^{-1}\,g(\lambda)\,
\\\notag
&&\quad
\times
(f^0(\lambda)+\lambda^{2d}g^0(\lambda))^{-1}\overline{{\me C}^{g0}_{\tau,N}(e^{i \lambda})}
d\lambda,
\end{eqnarray*}
where
\begin{eqnarray*}
\me C^{f0}_{\tau,N}(e^{i\lambda})
&=&
\overline{g^0(\lambda)}\vec{A}_{\tau,N}(e^{i\lambda}) +
\vec{C}_{\tau,N}^0(e^{i \lambda}),
\\
{\me C}^{g0}_{\tau,N}(e^{i \lambda})
&=&
{|1-e^{i\lambda \tau}|^{2n}}\lambda^{-2n}\overline{f^0(\lambda)}\vec{A}_N(e^{i\lambda})
-{(1-e^{-i\lambda \tau})^n}
\vec{C}_{\tau,N}^0(e^{i \lambda}).
\end{eqnarray*}
and
\[
\vec{C}^0_{\tau,N}(e^{i \lambda})=\sum_{j=0}^{N+\tau d}
(((\me P^{\tau}_N)^0)^{-1}[D_N^{\tau}\me a_N]_{+\tau d}-((\me P^{\tau}_N)^0)^{-1}(\me T^{\tau}_N)^0\me a^{\tau}_N)_j e^{i\lambda j}.
\]

The constrained optimisation problem (\ref{cond-extr1_A_N_i_n_c}) is equivalent to the unconstrained optimisation problem
\be  \label{uncond-extr_A_N_i_n_c}
\Delta_{\mathcal{D}}(f,g)=\widetilde{\Delta}(f,g)+ \delta(f,g|\mathcal{D}_f\times
\mathcal{D}_g)\to\inf,\ee
 where $\delta(f,g|\mathcal{D}_f\times
\mathcal{D}_g)$ is the indicator function of the set
$\mathcal{D}=\mathcal{D}_f\times\mathcal{D}_g$.
 Solution $(f^0,g^0)$ to this unconstrained optimisation problem is characterized by the condition $0\in
\partial\Delta_{\mathcal{D}}(f^0,g^0)$, where
$\partial\Delta_{\mathcal{D}}(f^0,g^0)$ is the subdifferential of the functional $\Delta_{\mathcal{D}}(f,g)$ at point $(f^0,g^0)\in \mathcal{D}=\mathcal{D}_f\times\mathcal{D}_g$.
 This condition makes it possible to find the least favourable spectral densities in some special classes of spectral densities $\mathcal{D}=\mathcal{D}_f\times\mathcal{D}_g$.

The form of the functional $\Delta(h_{\tau,N}(f^0,g^0);f,g)$ is convenient for application the Lagrange method of indefinite multipliers for
finding solution to the problem (\ref{uncond-extr_A_N_i_n_c}).
Making use of the method of Lagrange multipliers and the form of
subdifferentials of the indicator functions $\delta(f,g|\mathcal{D}_f\times
\mathcal{D}_g)$ of the set
$\mathcal{D}_f\times\mathcal{D}_g$ of spectral densities
we describe relations that determine least favourable spectral densities in some special classes
of spectral densities (see \cite{Moklyachuk2015,Luz_Mokl_book} for additional details).

\subsection{Least favorable spectral density in classes $\md D_0 \times \md \md D_{V}^{U}$}\label{set1_A_N_i_n_c}

Consider the problem of optimal linear estimation of the functional $A_{NT}{\xi}$
 which depends on unobserved values of the process $ \xi(t)$ with periodically stationary increments based on observations of the process $ \xi(t)+ \eta(t)$
 at points   $\mr R\setminus[0;(N+1)T]$
  under the condition that the sets of admissible spectral densities $\md D_{f0}^k, \md D_{Vg}^{Uk},k=1,2,3,4$ are defined as follows:
\begin{eqnarray*}
\md D_{f0}^{1} &=& \bigg\{f(\lambda )\left|\frac{1}{2\pi} \int
_{-\pi}^{\pi}
\frac{|1-e^{i\lambda\tau}|^{2d}}{\lambda^{2d}}
f(\lambda )d\lambda  =P\right.\bigg\},
\\
 \md D_{f0}^{2} &=&\bigg\{f(\lambda )\left|\frac{1}{2\pi }
\int _{-\pi }^{\pi}
\frac{|1-e^{i\lambda\tau}|^{2d}}{\lambda^{2d}}
{\rm{Tr}}\,[ f(\lambda )]d\lambda =p\right.\bigg\},
\\
\md D_{f0}^{3} &=&\bigg\{f(\lambda )\left|\frac{1}{2\pi }
\int _{-\pi}^{\pi}
\frac{|1-e^{i\lambda\tau}|^{2d}}{\lambda^{2d}}
f_{kk} (\lambda )d\lambda =p_{k}, k=\overline{1,\infty}\right.\bigg\},
\\
\md D_{f0}^{4} &=&\bigg\{f(\lambda )\left|\frac{1}{2\pi} \int _{-\pi}^{\pi}
\frac{|1-e^{i\lambda\tau}|^{2d}}{\lambda^{2d}}
\left\langle B_{1} ,f(\lambda )\right\rangle d\lambda  =p\right.\bigg\},
\end{eqnarray*}
and
\begin{eqnarray*}
 {\md D_{Vg}^{U1}}&=&\left\{g(\lambda )\bigg|V(\lambda )\leq g(\lambda
)\leq U(\lambda ), \frac{1}{2\pi } \int _{-\pi}^{\pi}
g(\lambda )d\lambda=Q\right\},
\\
  {\md D_{Vg}^{U2}} &=&\bigg\{g(\lambda )\bigg|{\mathrm{Tr}}\, [V(\lambda
)]\leq {\mathrm{Tr}}\,[ g(\lambda )]\leq {\mathrm{Tr}}\, [U(\lambda )],
\frac{1}{2\pi } \int _{-\pi}^{\pi}
{\mathrm{Tr}}\,  [g(\lambda)]d\lambda  =q \bigg\},
\\
{\md D_{Vg}^{U3}} &=&\bigg\{g(\lambda )\bigg|v_{kk} (\lambda )  \leq
g_{kk} (\lambda )\leq u_{kk} (\lambda ),
\frac{1}{2\pi} \int _{-\pi}^{\pi}
g_{kk} (\lambda
)d\lambda  =q_{k} , k=\overline{1,\infty}\bigg\},
\\
{\md D_{Vg}^{U4}} &=&\bigg\{g(\lambda )\bigg|\left\langle B_{2}
,V(\lambda )\right\rangle \leq \left\langle B_{2},g(\lambda
)\right\rangle \leq \left\langle B_{2} ,U(\lambda)\right\rangle,
\frac{1}{2\pi }
\int _{-\pi}^{\pi}
\left\langle B_{2},g(\lambda)\right\rangle d\lambda  =q\bigg\}.
\end{eqnarray*}

\noindent
Here  $V( \lambda )$, $U( \lambda )$ are known and fixed spectral densities, $W(\lambda)$ is an unknown spectral density,
 $p, p_k, k=\overline{1,\infty}$, $q,q_{k},k=\overline{1,\infty}$,   are given numbers, $P, Q, B_1, B_2$ are given positive-definite Hermitian matrices.

From the condition $0\in\partial\Delta_{\mathcal{D}}(f^0,g^0)$
we find the following equations which determine the least favourable spectral densities for these given sets of admissible spectral densities.

For the first set of admissible spectral densities $\md D_{f0}^1 \times \md D_{Vg}^{U1}$ we have equations
\begin{multline} \label{eq_4_1f_A_N_i_n_c}
\left(
\me C^{f0}_{\tau,N}(e^{i\lambda})
\right)
\left(
\me C^{f0}_{\tau,N}(e^{i\lambda})
\right)^{*}
=\left(\frac{|1-e^{i\lambda\tau}|^{2d}}{\lambda^{2d}} (f^0(\lambda)+\lambda^{2d}g^0(\lambda))\right)
\\
\times \vec{\alpha}_f\cdot \vec{\alpha}_f^{*}\left(\frac{|1-e^{i\lambda\tau}|^{2d}}{\lambda^{2d}} (f^0(\lambda)+\lambda^{2d}g^0(\lambda))\right),
\end{multline}
\begin{multline} \label{eq_4_1g_A_N_i_n_c}
\left(
\me C^{g0}_{\tau,N}(e^{i\lambda})
\right)
\left(
\me C^{g0}_{\tau,N}(e^{i\lambda})
\right)^{*}
=
\left(\frac{|1-e^{i\lambda\tau}|^{2d}}{\lambda^{2d}}(f^0(\lambda)+{\lambda}^{2n}g^0(\lambda))\right)
\\
\times
(\vec{\beta}\cdot \vec{\beta}^{*}+\Gamma _{1} (\lambda )+\Gamma _{2} (\lambda ))\left(\frac{|1-e^{i\lambda\tau}|^{2d}}{\lambda^{2d}}(f^0(\lambda)+{\lambda}^{2n}g^0(\lambda))\right),
\end{multline}

\noindent where $\vec{\alpha}_f$ and $ \vec{\beta}$  are vectors of Lagrange multipliers,
the matrix $\Gamma _{1} (\lambda )\le 0$ and $\Gamma _{1} (\lambda )=0$ if $g_{0}(\lambda )>V(\lambda ),$ the matrix  $
\Gamma _{2} (\lambda )\ge 0$ and $\Gamma _{2} (\lambda )=0$ if $g_{0}(\lambda )<U(\lambda )$.

For the second set of admissible spectral densities $\md D_{f0}^2 \times \md D_{Vg}^{U2}$ we have equation
\be \label{eq_4_2f_A_N_i_n_c}
\left(
\me C^{f0}_{\tau,N}(e^{i\lambda})
\right)
\left(
\me C^{f0}_{\tau,N}(e^{i\lambda})
\right)^{*}
=\alpha_f^{2} \left(\frac{|1-e^{i\lambda\tau}|^{2d}}{\lambda^{2d}} (f^0(\lambda)+\lambda^{2d}g^0(\lambda))\right)^2,
\ee
\begin{multline}\label{eq_4_2g_A_N_i_n_c}
\left(
\me C^{g0}_{\tau,N}(e^{i\lambda})
\right)
\left(
\me C^{g0}_{\tau,N}(e^{i\lambda})=
\right)^{*}
\\
=(\beta^{2} +\gamma _{1} (\lambda )+\gamma _{2} (\lambda )) \left(\frac{|1-e^{i\lambda\tau}|^{2d}}{\lambda^{2d}}(f^0(\lambda)+{\lambda}^{2n}g^0(\lambda))\right)^2,
\end{multline}

\noindent where $\alpha _{f}^{2}$, $\beta^{2}$ are Lagrange multipliers,  the
 function $\gamma _{1} (\lambda )\le 0$ and $\gamma _{1} (\lambda )=0$ if ${\mathrm{Tr}}\,
[g_{0} (\lambda )]> {\mathrm{Tr}}\,  [V(\lambda )],$ the function $\gamma _{2} (\lambda )\ge 0$ and $\gamma _{2} (\lambda )=0$ if $ {\mathrm{Tr}}\,[g_{0}(\lambda )]< {\mathrm{Tr}}\, [ U(\lambda)]$.

For the third set of admissible spectral densities $\md D_{f0}^{3}\times \md D_{Vg}^{U3}$ we have equation
\begin{multline} \label{eq_4_3f_A_N_i_n_c}
\left(
\me C^{f0}_{\tau,N}(e^{i\lambda})
\right)
\left(
\me C^{f0}_{\tau,N}(e^{i\lambda})
\right)^{*}
=\left(\frac{|1-e^{i\lambda\tau}|^{2d}}{\lambda^{2d}} (f^0(\lambda)+\lambda^{2d}g^0(\lambda))\right)
\\
\times
\left\{\alpha _{fk}^{2} \delta _{kl} \right\}_{k,l=1}^{\infty}
\left(\frac{|1-e^{i\lambda\tau}|^{2d}}{\lambda^{2d}} (f^0(\lambda)+\lambda^{2d}g^0(\lambda))\right),
\end{multline}
\begin{multline} \label{eq_4_3g_A_N_i_n_c}
\left(
\me C^{g0}_{\tau,N}(e^{i\lambda})
\right)
\left(
\me C^{g0}_{\tau,N}(e^{i\lambda})
\right)^{*}
=\left(\frac{|1-e^{i\lambda\tau}|^{2d}}{\lambda^{2d}}(f^0(\lambda)+{\lambda}^{2n}g^0(\lambda))\right)
\\
\times
\left\{(\beta_{k}^{2} +\gamma _{1k} (\lambda )+\gamma _{2k} (\lambda ))\delta _{kl}\right\}_{k,l=1}^{\infty}\left(\frac{|1-e^{i\lambda\tau}|^{2d}}{\lambda^{2d}}(f^0(\lambda)+{\lambda}^{2n}g^0(\lambda))\right),
\end{multline}

\noindent where $\alpha_{fk}^{2}$,  $\beta_{k}^{2}$ are Lagrange multipliers,
$\delta _{kl}$ are Kronecker symbols,
the
functions $\gamma _{1k} (\lambda )\le 0$ and $\gamma _{1k} (\lambda )=0$ if $g_{kk}^{0} (\lambda )>v_{kk} (\lambda ),$ the functions $\gamma _{2k} (\lambda )\ge 0$ and $\gamma _{2k} (\lambda )=0$ if $g_{kk}^{0} (\lambda )<u_{kk} (\lambda).$

For the fourth set of admissible spectral densities $\md D_{f0}^4 \times\md D_{Vg}^{U4}$ we have equations
\begin{multline} \label{eq_4_4f_A_N_i_n_c}
\left(
\me C^{f0}_{\tau,N}(e^{i\lambda})
\right)
\left(
\me C^{f0}_{\tau,N}(e^{i\lambda})
\right)^{*}
=
\alpha_f^{2} \left(\frac{|1-e^{i\lambda\tau}|^{2d}}{\lambda^{2d}} (f^0(\lambda)+\lambda^{2d}g^0(\lambda))\right)
\\
\times B_{1}^{\top}
\left(\frac{|1-e^{i\lambda\tau}|^{2d}}{\lambda^{2d}} (f^0(\lambda)+\lambda^{2d}g^0(\lambda))\right),
\end{multline}
\begin{multline} \label{eq_4_4g_A_N_i_n_c}
\left(
\me C^{g0}_{\tau,N}(e^{i\lambda})
\right)
\left(
\me C^{g0}_{\tau,N}(e^{i\lambda})
\right)^{*}
\\
=
(\beta^{2} +\gamma'_{1}(\lambda )+\gamma'_{2}(\lambda ))
\left(\frac{|1-e^{i\lambda\tau}|^{2d}}{\lambda^{2d}}(f^0(\lambda)+{\lambda}^{2n}g^0(\lambda))\right)
\\
\times
B_{2}^{\top}\left(\frac{|1-e^{i\lambda\tau}|^{2d}}{\lambda^{2d}}(f^0(\lambda)+{\lambda}^{2n}g^0(\lambda))\right),
\end{multline}

\noindent where  $\alpha _{f}^{2}$, $ \beta^{2}$, are Lagrange multipliers,
the
functions $\gamma'_{1}( \lambda )\le 0$ and $\gamma'_{1} ( \lambda )=0$ if $\langle B_{2},g_{0} ( \lambda) \rangle > \langle B_{2},V( \lambda ) \rangle,$ functions $\gamma'_{2}( \lambda )\ge 0$ and $\gamma'_{2} ( \lambda )=0$ if $\langle
B_{2} ,g_{0} ( \lambda) \rangle < \langle B_{2} ,U( \lambda ) \rangle.$

The following theorem  holds true.

\begin{thm}
Let   minimality condition (\ref{umova11_i_n_c}) hold true. The least favorable spectral densities $f_{0}(\lambda)$, $g_{0}(\lambda)$ in the classes
$\md D_{f0}^{k}\times \md D_{Vg}^{Uk},k=1,2,3,4$ for the optimal linear estimation of the functional  $A_N\vec{\xi}$ from observations of the process $ {\xi}(t)+  {\eta}(t)$ at points
$\mr R\setminus[0;(N+1)T]$
are determined by equations
\eqref{eq_4_1f_A_N_i_n_c}--\eqref{eq_4_1g_A_N_i_n_c},  \eqref{eq_4_2f_A_N_i_n_c}--\eqref{eq_4_2g_A_N_i_n_c}, \eqref{eq_4_3f_A_N_i_n_c}--\eqref{eq_4_3g_A_N_i_n_c}, \eqref{eq_4_4f_A_N_i_n_c}--\eqref{eq_4_4g_A_N_i_n_c},
respectively,
the constrained optimization problem (\ref{minimax1_A_N_i_n_c}) and restrictions  on densities from the corresponding classes
$ \md D_{f0}^{k}, \md D_{Vg}^{Uk},k=1,2,3,4$.  The minimax-robust spectral characteristic of the optimal estimate of the functional $A_{NT}{\xi}$ is determined by   formula (\ref{spectr_A_N_i_n_c}).
\end{thm}

\section{Conclusions}

In this article, we present results of investigation of the stochastic processes with periodically stationary increments.
These non-stationary stochastic processes combine  periodic structure of covariation functions of sequences as well as integrating one.

We describe methods of solution of the classical interpolation problem for a linear functional   constructed from unobserved values of a process  with periodically stationary increments.
Estimates are based on observations of this process with a periodically stationary noise process.
Estimates are obtained by representing the process under investigation as a generated infinite dimensional vector  sequence with stationary increments.
The problem is investigated in the case of spectral certainty, where spectral densities of the generates sequences are exactly known.
In this case,  we propose an approach based on the Hilbert space projection method.
We derive formulas for calculating the spectral characteristics and the mean-square errors of the optimal estimates of the functionals.
In the case of spectral uncertainty where the spectral densities are not exactly known while, instead, some sets of admissible spectral densities are specified,
the minimax-robust method is applied.
We propose a representation of the mean square error in the form of a linear
functional in $L_1$ space with respect to spectral densities, which allows
us to solve the corresponding constrained optimization problem and
describe the minimax-robust estimates of the functionals. Formulas
that determine the least favorable spectral densities and minimax-robust spectral characteristic of the optimal linear estimates of
the functionals are derived  for a collection of specific classes
of admissible spectral densities.

\end{document}